\documentclass[12pt]{amsart}

\usepackage[all]{xy}
\usepackage{amsmath,amssymb,mathrsfs,amsthm}
\newcommand{\ms}[2]{\mathscr{#1}\mathit{#2}}
\renewcommand{\sc}[1]{\mathscr{#1}}
\newcommand{\mbb}[1]{\mathbb{#1}}

\renewcommand{\u}[1]{\underline{#1}}
\renewcommand{\t}[1]{\tilde{#1}}

\numberwithin{equation}{section}
\newtheorem{theorem}{Theorem}
\newtheorem{proposition}[theorem]{Proposition}
\newtheorem{lemma}[theorem]{Lemma}
\newtheorem{corollary}[theorem]{Corollary}

\newtheorem*{theorem*}{Theorem}
\newtheorem*{corollary*}{Corollary}
\newtheorem*{proposition*}{Proposition}
\theoremstyle{definition}
\newtheorem{remark}[theorem]{Remark}

\newtheorem*{question*}{Question}

\begin{document}
\title{Smash nilpotent cycles on products of curves}

\author[R. Sebastian]{Ronnie Sebastian}

\address{Humboldt Universit\"{a}t zu Berlin, 
Institut f\"{u}r Mathematik, 
Rudower Chaussee 25, 
10099 Berlin}

\email{ronnie.sebastian@gmail.com}

\begin{abstract}
Voevodsky has conjectured that numerical and smash equivalence coincide on a smooth projective variety. We prove the conjecture 
for one dimensional cycles on an arbitrary product of curves. As a consequence we get that numerically trivial 1-cycles on 
an abelian variety are smash nilpotent.
\end{abstract}

\maketitle

\section{Introduction}
Throughout this article we work over an algebraically closed field. Let $X$ be a smooth and projective variety. Define  
rational cycles of codimension $n$ to be elements of the group $Z^n(X):=\bigoplus_{Y\subset X}\mathbb{Q}\cdot [Y]$, 
where $Y$ varies over the set of irreducible closed subvarieties of $X$ of codimension $n$. Various equivalence relations may be defined on the group of cycles. 
Examples of such equivalence relations are rational equivalence, algebraic equivalence, homological 
equivalence and numerical equivalence.

Homological equivalence is defined using a Weil cohomology theory. An important consequence of the two standard conjectures (Lefschetz standard 
conjecture and Hodge standard conjecture) is that numerical and homological equivalence coincide, which would mean that being homologically 
trivial is independent of the choice of the Weil cohomology theory. For an exposition on the standard conjectures we refer the reader to 
\cite{kleiman}. In \cite{kleiman}, Kleiman writes, ``Sometimes, the coincidence of the two relations may be proved by sandwiching homological 
equivalence between numerical equivalence and another equivalence relation.''

Voevodsky \cite{voe} introduced the adequate relation of smash nilpotence. Let $X$ be a smooth projective variety over a field $k$. An algebraic
cycle $\alpha$ on $X$ (with rational coefficients) is \emph{smash-nilpotent} if there exists $n>0$
such that $\alpha^n$ is rationally equivalent to $0$ on $X^n$. 
Voevodsky \cite[Cor. 3.3]{voe} and Voisin \cite[Lemma 2.3]{voisin} proved that any cycle algebraically equivalent to $0$ is
smash-nilpotent. Because of the multiplicative property of the cycle class map, any smash-nilpotent 
cycle is homologically equivalent to 0 and so numerically equivalent to $0$; Voevodsky conjectured that the converse is true \cite[Conj.
4.2]{voe}. This would imply that homological and numerical equivalence coincide. This also implies Bloch's conjecture on zero cycles on surfaces 
with $p_g=0$. Moreover, this is an interesting question in itself.

The first general result which gave examples of smash nilpotent cycles is the following result of Kimura, \cite[Proposition 6.1]{kimura}. 
If $M$ and $N$ be finite dimensional motives with different parity and $f:M\to N$ a morphism of motives, then $f$ is smash nilpotent.

In \cite{shermenev}, Shermenev proves that, for an abelian variety $A$, the motive $h^1(A)$ is oddly finite dimensional. 
In \cite{ks}, the authors combine this with results in \cite{kimura}, to show that skew cycles 
($\beta$ is called \emph{skew} if $[-1]^*\beta=-\beta$) are smash nilpotent. From this, using results in \cite{beauvillechow}, they prove 
that on an abelian variety of dimension $\leq3$, any homologically trivial cycle is smash nilpotent. 
In fact, using the motivic hard Lefschetz theorem for abelian varieties, \cite[Theorem 5.2]{kunnemann}, one can easily 
improve the above theorem to equality of numerical and smash equivalence for abelian varieties of dimension $\leq 3$. 
In particular, Voevodsky's conjecture is true for abelian varieties of dimension $\leq3$.

Conjecturally, morphisms between motives of the same parity given by a numerically trivial cycle should be smash nilpotent. 
The author is not aware of any nontrivial examples (that is, not algebraically equivalent 
to 0 or in the subalgebra generated by skew cycles on an abelian variety under the operations of intersection product and Pontryagin 
product) or any general results in this direction. 

On abelian varieties of dimension $\geq4$, 
the Griffiths group of symmetric cycles ($\beta$ is called \emph{symmetric} if $[-1]^*\beta=\beta$) can have infinite rank, 
as shown by Fakhruddin in \cite[Theorem 4.4]{fakhruddin}. Symmetric cycles can be viewed as morphisms between motives of the same parity. 
The methods in \cite{kimura} and \cite{ks} do not directly apply to symmetric 
cycles and one is interested in knowing if these cycles are smash nilpotent. In this article we give the first fairly general results 
in this direction. 

The main theorem in this article is 
\begin{theorem*}
Voevodsky's conjecture is true for 1-cycles on a product of curves.
\end{theorem*}
As a consequence of the above, we get the following results 
\begin{theorem*}
Voevodsky's conjecture is true for 1-cycles on abelian varieties. 
\end{theorem*}
In particular, the cycles constructed by Fakhruddin in \cite[Theorem 4.4, Corollary 4.6]{fakhruddin} are smash 
nilpotent. It is expected to be true that for any $i$, we can find $g$ sufficiently large such that the Beauville component $[C]_i$ 
is non-zero in the Griffiths group of the Jacobian of a generic curve of genus $g$. 
\begin{corollary*}
Numerical and smash equivalence coincide on the tautological subring of the Chow ring modulo algebraic equivalence of the 
Jacobian of a smooth projective curve.
\end{corollary*}

The proof of theorem \ref{mt} proceeds by induction on the number of factors in the product. 
For any 1-cycle $\alpha$, we show $\alpha\sim_{sm}\sum\alpha_i$ such that each $\alpha_i$ is 
obtained in a canonical way from $\alpha$ and comes from a smaller product of curves. 
To prove this assertion, we are lead to considering 1-cycles on $C^m$ of the type 
\[\Delta_e=\sum_{T\neq\emptyset, T\subset S}q_{\#T}\Delta_T\]
Here $S$ denotes the set $\{1,2,\ldots,m\}$ and $\Delta_T$ denotes the curve $C$ embedded diagonally into the factor given by $T$. 
The above ``modified diagonal'' cycle is inspired from \cite{schoen}. If $m$ is small, then it is not clear whether this cycle is 
smash nilpotent. However, for $m\gg0$, since this is a symmetric cycle and $S^mC$ is a projective bundle over $J(C)$, it is easy to deduce sufficient 
conditions on the rational coefficients so that 
\begin{itemize}
	\item $\Delta_e$ is smash nilpotent.
	\item Projecting $\Delta_e$ to a smaller product $C^n$ yields a nontrivial relation of the type $\alpha\sim_{sm}\alpha_i$.
\end{itemize}

Solving for the coefficients boils down to showing 
that certain linear homogeneous polynomials are linearly independent, which is done in section 4. 

It was brought to our attention by N. Fakhruddin, that the above method can be used to prove the following interesting 
proposition. The idea is similar to the above, if we have a symmetric 1-cycle
on $C^m$ then we can map it to $J(C)$ in the usual way. Another thing one could do is that we can first project it to $C^n$ for some $n<m$
and then again push it forward to $J(C)$. Then these two cycles in $J(C)$ can be very
different; the
first could be something simple (that is, something one understands) whereas the
second could be very different, that is, not expressible in some elementary way in
terms of the first.
\begin{proposition*}
Let $C$ be a smooth projective curve. Denote by $\alpha_j\in CH^{g-1}_j(J(C))/\sim_{alg}$ its Beauville components. If $\alpha_i=0$ for 
some $i\geq0$, then $\alpha_j=0$ for every $j>i$.
\end{proposition*}

\vskip 3mm
\noindent {\bf Acknowledgements.} 
I am very grateful to N. Fakhruddin for several useful suggestions.
He pointed out to me proposition \ref{p7} and his comments lead to a simplification of the proof of lemma \ref{soleq}, 
among other improvements. I would like to 
thank J. Biswas, D.S. Nagaraj and V. Srinivas for useful discussions, and Prof. J.P. Murre for his generous encouragement 
and for his lectures at the Tata Institute, Mumbai, from where I learnt about these problems.
This work was done during my visit to the Institut f\"{u}r Mathematik at the Humboldt Universit\"{a}t zu Berlin,
made possible by a grant by the IMU Berlin Einstein Foundation. I would like to thank the institute for its hospitality.

\setcounter{theorem}{0}
\section{1-cycles on products of curves}
For a smooth projective curve $C$, we denote its Jacobian by 
$J(C)$. In this section we show that numerically trivial 1-cycles on products of smooth projective curves are smash nilpotent.

Let $\t{D}$ be a smooth projective curve of genus $g$ and let $m$ be an integer such that $m>2g+2$. 
Let $S$ denote the set $\{1,2,\ldots,m\}$. Let $p_i:\t{D}^m\to \t{D}$ denote the projection onto the $i$-th factor. 
For every nonempty $T\subset S$ consider the morphism 
$\phi_T:\t{D}\to \t{D}^m$ which is defined by 
\[ p_i\circ\phi_T(d)=\left\{\begin{array}{c} d\qquad\qquad \text{for } i\in T\\
				d_0 \qquad\qquad \text{for } i\notin T
                    \end{array}\right\}
\]
and define $\Delta_T$ to be $\phi_{T*}(\t{D})$. 

The modified diagonal cycle was introduced in \cite{schoen}, who show that it is homologically trivial. 
We define a more general modified diagonal cycle in the Chow group of $\t{D}^m$ to be 
\begin{equation}\label{nmoddiag}
\Delta_e=\sum_{T\neq\emptyset, T\subset S}q_{\#T}\Delta_T
\end{equation}
We want to choose the coefficients $q_i,i=1,2\ldots,m$, appropriately so that $\Delta_e$ becomes smash nilpotent.
Consider the following two conditions 
\begin{enumerate}
	\item[(S1)] For every {\bf even} integer $i$ in the set $\{0,1,2,\ldots,g-1\}$,
		\[\Bigg(\sum_{k=1}^mq_k\left(\begin{array}{c}m\\k\end{array}\right)k^{2+i}\Bigg)=0\]
	\item[(S2)] $\sum_{k=1}^mq_k\left(\begin{array}{c}m\\k\end{array}\right)k=0$
\end{enumerate}
\begin{remark}The above is a homogeneous system of equations in $q_k$'s and since $m>2g+2$, there are more variables than equations. 
Thus, it will always have nontrivial solutions. 
\end{remark}

Let $f:\t{D}^{m}\to S^m\t{D}$ denote the quotient by the group $S_m$, which acts on $\t{D}^{m}$ by permutations. 
On $S^m\t{D}$ consider the (reduced) subvariety $\Delta^s_i:=f(\Delta_T)$, for some $T$ with $\#T=i$ (this is clear depends 
only on $\#T$). 
\begin{proposition}\label{p1}We have the following relation in $CH_1(\t{D}^m)$
\[\frac{1}{i!(m-i)!}f^*\Delta^s_i=\sum_{\#T=i}\Delta_T\]
\end{proposition}
\begin{proof}
Since $f:\t{D}^m\to S^m\t{D}$ is a finite flat Galois morphism, it is clear that $f^*\Delta^s_i$ is a multiple of $\sum_{\#T=i}\Delta_T$.
Applying $f_*$ on both cycles and using the fact that $f_*f^*=m!$ for a finite map, we get the desired result.
\end{proof}
Since $m>2g+2$, $S^m\t{D}$ is the projective bundle associated to a locally free sheaf on $J(\t{D})$. We may take 
$\sc{O}(1)$ to be the line bundle associated to the (reduced) divisor $f(p_1^{-1}(d_1))$.
\begin{proposition}\label{p2}
$deg(f_*(\Delta_T)\cap c_1(\sc{O}(1)))=(m-1)!(\#T)$
\end{proposition}
\begin{proof}
It suffices to compute the degree of $\Delta_T\cap f^*c_1(\sc{O}(1))$. Arguing as in the proof of proposition \ref{p1}, 
we get $f^*(c_1(\sc{O}(1)))=(m-1)!\sum_{i=1}^mp_i^{-1}(d_1)$. Taking the set theoretic intersection of $\Delta_T$ with $\sum_{i=1}^mp_i^*\sc{O}_{\t{D}}(d_1)$ 
gives the desired result.
\end{proof}

The Beauville decomposition of the Chow group of 1-cycles on 
$J(\t{D})$ is given by (see \cite[Proposition 3a]{beauvillechow})
\[
CH^{g-1}_{\mbb{Q}}(J(\t{D}))=\bigoplus_{s=0}^{g-1}CH^{g-1}_s(J(\t{D}))\\
\]
\[CH^{g-1}_s(J(\t{D}))=\{x\in CH^{g-1}_{\mbb{Q}}(J(\t{D}))\Big\vert[n]_*(x)=n^{2+s}x\}\]
Write the image of the curve embedded into $J(\t{D})$ using the base point $d_0$ as
$[\t{D}]=\sum_{i=0}^{g-1}\alpha_i$ where $\alpha_i$ satisfies 
$[n]_*\alpha_i=n^{2+i}\alpha_i$.

\begin{lemma}\label{l1}
Let $q_i,i=1,2,\ldots,m$ be rational numbers satisfying conditions (S1) and (S2). Then the cycle $\Delta_e$ in equation \eqref{nmoddiag} is smash nilpotent.
\end{lemma}
\begin{proof}
Define a 1-cycle on $S^m\t{D}$ by 
\[\Gamma_e:=\sum_{k=1}^mq_k\left(\begin{array}{c}m\\k\end{array}\right)\Delta_k^s\]
From proposition \ref{p1} it is clear that $f^*(\Gamma_e)=m!\Delta_e$ and it suffices to show that $\Gamma_e$ is smash nilpotent.
Using the base point $d_0$ get a map $p:S^m\t{D}\to J(\t{D})$. Since $m>2g-1$, $p$ is a 
projective bundle associated to a locally free sheaf. In particular, the cycle $\Gamma_e$, which is a 1-cycle, may be written as, 
see \cite[Theorem 3.3(b), Proposition 3.1(a)(i)]{fulton},
\[\Gamma_e=c_1(\ms{O}{}(1))^{m-g-1}\cap p^*\beta_0\oplus c_1(\ms{O}{}(1))^{m-g}\cap p^*\beta_1\] 
In the above equation, $\beta_i\in CH_i(J(\t{D}))$ are given by 
\begin{equation}\beta_0=p_*(\Gamma_e\cap c_1(\ms{O}{}(1)))-p_*(c_1(\ms{O}{}(1))^{m-g+1}\cap p^*\beta_1)\qquad\text{and}\qquad \beta_1=p_*(\Gamma_e)\end{equation}
First we show $\beta_1$ is smash nilpotent.  
\begin{align*}\beta_1&=\sum_{k=1}^mq_k\left(\begin{array}{c}m\\k\end{array}\right)p_*(\Delta_k^s)=
		\sum_{k=1}^mq_k\left(\begin{array}{c}m\\k\end{array}\right)[k]_*([\t{D}])\\
		&=\sum_{k=1}^mq_k\left(\begin{array}{c}m\\k\end{array}\right)[k]_*\Big(\sum_{i=0}^{g-1}\alpha_i\Big)\\
			&=\sum_{i=0}^{g-1}\Bigg(\sum_{k=1}^mq_k\left(\begin{array}{c}m\\k\end{array}\right)k^{2+i}\Bigg)\alpha_i
\end{align*}
Using \cite[Proposition 1]{ks}, $\alpha_i$ is smash nilpotent for $i$ odd. Since the $q_k$'s satisfy (S1),
it follows that $\beta_1$ is smash nilpotent. In particular, $\beta_1$ is numerically trivial.

Next we compute the degree of $\beta_0$. Since $\beta_1$ is numerically trivial,
\[deg(\beta_0)=deg(p_*(\Gamma_e\cap c_1(\ms{O}{}(1))))\]
A cycle of dimension 0 on an abelian variety is smash nilpotent if and only if its degree is 0. 
Using proposition \ref{p2}, it is easily checked that for $\beta_0$ to be smash nilpotent, we need that 
\[\sum_{k=1}^mq_k\left(\begin{array}{c}m\\k\end{array}\right)k=0\]
From this, it follows that $\Gamma_e$ is smash nilpotent and so $\Delta_e$ is 
smash nilpotent. 
\end{proof}

Let $C_i, i=1,2,\ldots,n$ be smooth, projective curves. Denote by $X$ the product $C_1\times C_2\times \ldots \times C_n$. 
Let $p'_i:X\to C_i$ denote the projection onto the $i$-th factor.
Let $j:D\hookrightarrow X$ be an irreducible curve and let $\phi_n:\t{D}\to D$ denote its normalization. Denote the composite 
$\phi_n\circ j$ by $\t{j}:\t{D}\to X$. Let $g$ denote the genus of $\t{D}$. Let $m$ be an integer such that $m>n,2g+2$.

Define a morphism $\psi:\t{D}^m\to X$ by projecting onto the first $n$ coordinates as follows
\begin{align*}\psi:\t{D}^m\to X\qquad\qquad\psi:=(p'_1\circ\t{j})\times (p'_2\circ\t{j})\times\ldots\times (p'_n\circ\t{j})\end{align*}
The morphisms $\t{j}$ factor as  
\begin{equation}\label{npsi}\xymatrix{&\t{D}^m\ar[d]^{\psi}&  \\
	\t{D}\ar[r]_{\t{j}}\ar[ur]^{\Delta}  & X }
\end{equation}
The 1-cycle $\psi_*(\Delta_e)$ on $X$ is smash nilpotent. We want to understand what this 1-cycle looks like. 
Write the set $S$ as a disjoint union 
\[S=S_0\sqcup S'\qquad S_0=\{1,2,\ldots, n\}\qquad S'=\{n+1,\ldots,m\}\]
By $\u{x}$ we will denote closed points of $X$. For $T\subset S_0$ and $\u{c}$, define 
\[\zeta_T^{\u{c}}:X\to X\]
as follows
\[ p_i\circ \zeta_T^{\u{c}}(\u{x})=\left\{\begin{array}{c} x_i\qquad\qquad \text{for } i\in T\\
				c_i \qquad\qquad \text{for } i\notin T
                    \end{array}\right\}
\]
For example, if $T=S_0$, then $\zeta_T^{\u{c}}$ is the identity. In words, $\zeta_T^{\u{c}}$ is the composite of 
a projection onto the $T$ coordinates and then an inclusion into $X$ using base points given by $\u{c}$.
\begin{remark}\label{rem2}
It is clear that if $\u{v}, \u{w}\in X$ are two closed points, then for any cycle $\alpha$, the cycles $\zeta^{\u{v}}_{T*}(\alpha)$ and 
$\zeta^{\u{w}}_{T*}(\alpha)$ are algebraically equivalent.
\end{remark} 
\noindent Observe that $\psi_*(\Delta_T)=0$ if $T\cap S_0=\emptyset$. 

\noindent Define $\u{v}$ by $v_i=p'_i\circ \t{j}(d_0)$.
It follows that  
\begin{align*}
	\psi_*(\Delta_e)&=\psi_*(\sum_{T\neq\emptyset, T\subset S}q_{\#T}\Delta_T)=\sum_{T\neq\emptyset, T\subset S}q_{\#T}\psi_*(\Delta_T)\\
				&=\sum_{T\subset S, T\cap S_0\neq\emptyset }q_{\#T}\zeta_{(T\cap S_0)*}^{\u{v}}(\t{j}_*(\t{D}))
\end{align*}
One has $T\cap S_0=S_0$ if and only if $S_0\subset T$. 
Let us denote by $\sc{U}$ the collection of subsets of $S$ having this property and let us denote  by $\sc{S}$ 
the collection of subsets of $S$ with the property $T\cap S_0\neq\emptyset$. Then we have
\begin{align}
	\psi_*(\Delta_e)&=\sum_{T\in \sc{U}}q_{\#T}\zeta_{(T\cap S_0)*}^{\u{v}}(\t{j}_*(\t{D}))+\sum_{T\in \sc{S}\setminus\sc{U}}q_{\#T}\zeta_{(T\cap S_0)*}^{\u{v}}(\t{j}_*(\t{D}))\nonumber\\
	&=\sum_{T\in \sc{U}}q_{\#T}[D]+\sum_{T\in \sc{S}\setminus\sc{U}}q_{\#T}\zeta_{(T\cap S_0)*}^{\u{v}}([D])\nonumber\\
	&=[D]\Big(\sum_{T\in \sc{U}}q_{\#T}\Big)+\sum_{T\in \sc{S}\setminus\sc{U}}q_{\#T}\zeta_{(T\cap S_0)*}^{\u{v}}([D])\nonumber\\
	&=[D]\Big(\sum_{i=0}^{m-n}\left(\begin{array}{c}m-n\\i\end{array}\right)q_{i+n}\Big)+\sum_{T\in \sc{S}\setminus\sc{U}}q_{\#T}\zeta_{(T\cap S_0)*}^{\u{v}}([D])\label{e3}
\end{align}
In equation \eqref{e3} we have put, for $T\in \sc{U}$, $\#(T\setminus S_0)=i$, and so $\#T=i+n$. For each $i$, there are exactly 
$\left(\begin{array}{c}m-n\\i\end{array}\right)$ choices of $T$'s in $\sc{U}$ with $\#T=i$.

In section \ref{se}, lemma \ref{soleq} it is shown for $n\geq3$, we can find integers $m>n,2g+2$ and rational numbers $q_i, i=1,2,\ldots,m$ 
satisfying (S1) and (S2) and  
\begin{enumerate}
	\item[(S3)] $\sum_{i=0}^{m-n}\left(\begin{array}{c}m-n\\i\end{array}\right)q_{i+n}\neq0$
\end{enumerate}

Let $\kappa:=\sum_{i=0}^{m-n}\left(\begin{array}{c}m-n\\i\end{array}\right)q_{i+n}$, then $\kappa\neq0$ and from the above we conclude 
that the following cycle on $X$ is smash nilpotent
\begin{equation}
\label{e6}[D] +\frac{1}{\kappa}\sum_{T\in \sc{S}\setminus\sc{U}}q_{\#T}\zeta^{\u{v}}_{(T\cap S_0)*}([D])
\end{equation}
Observe that the second term consists of cycles coming from a smaller product of curves.
We now prove the main theorem of this section 
\begin{theorem}\label{mt}
Numerical and smash equivalence coincide for 1-cycles on a product of curves.
\end{theorem}
\begin{proof}
Let $\alpha$ be a one dimensional cycle on $X$ such that 
\[\alpha=\sum_{i=1}^ta_iD_i\] 
where the $D_i$ are its distinct irreducible components.  
Let $\t{D}_i$ denote the normalization of $D_i$ and define as above
\[\t{j}_i:\t{D}_i\to X\] 
Choose a base point $d_i\in \t{D}_i$. Using the $d_i$ define points $\u{v}^i\in X$ 
by $v^i_j=p_j\circ \t{j}_i(d_i)$. Now choose a large integer $m>n,2g(\t{D}_i)+2$, and corresponding to this a collection 
of rational numbers such that (S1), (S2) and (S3) are satisfied. Carrying out the preceding discussion for 
each $i$, we get as in equation \eqref{e6} that the following cycle is smash nilpotent
\begin{align}
	[D_i] +\frac{1}{\kappa}\sum_{T\in \sc{S}\setminus\sc{U}}q_{\#T}\zeta^{\u{v}^i}_{(T\cap S_0)*}([D_i])
\end{align}
Multiplying with $a_i$ and summing over $i$, we get 
\[\alpha+\frac{1}{\kappa}\sum_{T\in \sc{S}\setminus\sc{U}}q_{\#T}\sum_{i=1}^ta_i\zeta^{\u{v}^i}_{(T\cap S_0)*}([D_i])\]
is smash nilpotent. Finally, going modulo algebraic equivalence and using remark \ref{rem2} we get the following cycle is 
smash nilpotent
\begin{align*}
	\alpha+\frac{1}{\kappa}\sum_{T\in \sc{S}\setminus\sc{U}}q_{\#T}\sum_{i=1}^ta_i\zeta^{\u{v}^i}_{(T\cap S_0)*}([D_i])&=\alpha+\frac{1}{\kappa}\sum_{T\in \sc{S}\setminus\sc{U}}q_{\#T}\sum_{i=1}^ta_i\zeta^{\u{v}^1}_{(T\cap S_0)*}([D_i])\\
		&=\alpha+\frac{1}{\kappa}\sum_{T\in \sc{S}\setminus\sc{U}}q_{\#T}\zeta^{\u{v}^1}_{(T\cap S_0)*}(\sum_{i=1}^ta_iD_i)\\
		&=\alpha+\frac{1}{\kappa}\sum_{T\in \sc{S}\setminus\sc{U}}q_{\#T}\zeta^{\u{v}^1}_{(T\cap S_0)*}(\alpha)
\end{align*}
Since $\alpha$ is numerically trivial it follows that $\zeta^{\u{v}^1}_{(T\cap S_0)*}(\alpha)$ is 
numerically trivial. By induction on $n$, since $\zeta^{\u{v}^1}_{(T\cap S_0)*}(\alpha)$ is the pushforward of a cycle from a smaller product of curves, we may 
assume that it is smash nilpotent. Thus, $\alpha$ being the sum of smash nilpotent cycles, is smash nilpotent. The base case for the induction is the product of 
two curves, which is a surface. On a surface numerical and algebraic equivalence coincide, and so numerical and smash equivalence 
coincide, see \cite[Cor. 3.3]{voe}.
\end{proof}

\begin{proposition}\label{p7}
Let $C$ be a smooth projective curve. Denote by $\alpha_j\in CH^{g-1}_j(J(C))/\sim_{alg}$ its Beauville components. If $\alpha_i=0$ for 
some $i\geq0$, then $\alpha_j=0$ for every $j>i$.
\end{proposition}
\begin{proof}It suffices to show that if $\alpha_i=0$ then $\alpha_{i+1}=0$.
Recall that the modified diagonal cycle on $C^m$ was defined as 
\[\Delta_e=\sum_{T\neq\emptyset, T\subset S}q_{\#T}\Delta_T\]
where $S=\{1,2,\ldots,m\}$. Since $\alpha_i=0$, arguing as in lemma \ref{l1}, it can be checked that $\Delta_e$ is algebraically trivial if 
\begin{equation}\label{nc}\sum_{k=1}^mq_k\left(\begin{array}{c}m\\k\end{array}\right)k^j=0\qquad\text{for }j\in\{1,2,\ldots,g-2\}\setminus{i} \end{equation}
Let $\psi:C^m\to C^n$ denote the projection onto the first $n$ factors, then 
\[\psi_*(\Delta_e)=\sum_{l=1}^n\Big(\sum_{j=0}^{m-n}\left(\begin{array}{c}m-n\\j\end{array}\right)q_{l+j}\Big)\cdot\Big(\sum_{\#W=l}\Delta_W\Big)\]
Pushing this cycle forward to the Jacobian, we get the cycle 
\begin{align}\sum_{l=1}^n\Big(\sum_{j=0}^{m-n}&\left(\begin{array}{c}m-n\\j\end{array}\right)q_{l+j}\Big)\left(\begin{array}{c}n\\l\end{array}\right)[l]_*([C])=\nonumber\\
&\sum_{l=1}^n\left(\begin{array}{c}n\\l\end{array}\right)\Big(\sum_{j=0}^{m-n}\left(\begin{array}{c}m-n\\j\end{array}\right)q_{l+j}\Big)\sum_{s=0}^{g-2}l^{2+s}\alpha_s\label{e10}\end{align}
If the above cycle is algebraically equivalent to 0, then for every $k$, one of the following two holds
\begin{enumerate}
	\item[(a)] $\alpha_s\sim_{alg}0$
	\item[(b)] $\sum_{l=1}^n\left(\begin{array}{c}n\\l\end{array}\right)\Big(\sum_{j=0}^{m-n}\left(\begin{array}{c}m-n\\j\end{array}\right)q_{l+j}\Big)l^{2+s}=0$
\end{enumerate}
Putting $s=i+1$ in the above, we may ask if we can find $m$ and rational numbers $q_k$ satisfying the following conditions
\begin{enumerate}
	\item $\sum_{k=1}^mq_k\left(\begin{array}{c}m\\k\end{array}\right)k^j=0\qquad\text{for }j\in\{1,2,\ldots,g-2\}\setminus{i}$
	\item $\sum_{l=1}^n\left(\begin{array}{c}n\\l\end{array}\right)\Big(\sum_{j=0}^{m-n}\left(\begin{array}{c}m-n\\j\end{array}\right)q_{l+j}\Big)l^{i+3}\neq0$
\end{enumerate}
This is possible is proved in lemma \ref{lem2}. 
\end{proof}
\begin{question*}
One may ask if the above proposition holds for higher dimensional subvarieties an abelian variety, that is, if $X\hookrightarrow A$ is a 
subvariety and if $[X]_i$ denote the Beauville components of $[X]$, then does $[X]_i\sim_{alg}0$ for some $i>0$ imply $[X]_j\sim_{alg}0$ 
for all $j>i$.
\end{question*}

\section{Cycles on Abelian varieties}
\begin{lemma}
Let $C$ be a smooth projective curve of genus $g$ and let $J(C)$ denote its Jacobian. Numerical and smash equivalence coincide for 1-cycles on $J(C)$.
\end{lemma}
\begin{proof}
Using the same notation as in the previous section, we have 
\[C^m\xrightarrow{f} S^mC\xrightarrow{p}J(C)\]
For $m>2g-2$, $p$ is a projective bundle associated to a locally free sheaf and so we have for any cycle $\alpha\in CH^*_{\mbb{Q}}(J(C))$
\[\alpha=\frac{1}{m!}p_*f_*f^*(p^*\alpha\cap c_1(\sc{O}(1))^{m-g})\]
Thus, if $\alpha$ is a numerically trivial 1-cycle, then $f^*(p^*\alpha\cap c_1(\sc{O}(1))^{m-g})$ is a numerically trivial 1-cycle 
on $C^m$ and so by the results of the previous section, this cycle is smash nilpotent and so $\alpha$ is smash nilpotent.
\end{proof}

\begin{theorem}\label{prop2}
Numerical and smash equivalence coincide for 1-cycles on abelian varieties.
\end{theorem}
\begin{proof}
Let $A$ be an abelian variety. Then we can find another abelian variety $B$ such that $A\times B$ is 
isogenous to the Jacobian of a curve. Let $\phi:A\times B\to J(C)$ and $\psi:J(C)\to A\times B$ be such that $\psi\circ \phi=[l]_{A\times B}$. 
Let $\alpha\in CH^{g_A-1}_s(A)$ 
be a numerically trivial 1-cycle. Consider the inclusion 
\[i:A\to A\times B\]
given by 
\[x\mapsto (x,e_B)\]
Since $p_1\circ i=1_A$, we get $p_{1*}\circ i_*=1_{CH(A)}$ and so $i_*$ is an inclusion. Now consider the commutative diagram
\[\xymatrix{A\ar[r]^i\ar[d]_{[n]} & A\times B\ar[d]^{[n]\times [n]}\\
		A\ar[r]_i& A\times B}
\]
Since \[([n]\times [n])_*i_*\alpha=i_*[n]_*\alpha=n^{2+s}i_*\alpha\]
we get that $i_*\alpha\in CH^{g_B+g_A-1}_s(A\times B)$ and is numerically trivial. Now $\psi^*(i_*\alpha)$ is numerically trivial and thus 
smash nilpotent, and so $\phi^*\psi^*(i_*\alpha)$ is smash 
nilpotent. This is $l^{2(g_B+g_A-1)-s}i_*\alpha$ and so we get $i_*\alpha$ is smash nilpotent. Pushing this forward to $A$, 
we get that $\alpha$ is smash nilpotent.
\end{proof}

\begin{proposition}\label{prop1}Let $A$ be an abelian variety of dimension $g$. 
Let $\theta\in CH^1_{\mbb{Q}}(A)$ be an ample symmetric class, that is, $[-1]^*\theta=\theta$. The multiplication map 
$\times \theta^{g-2p+s}:CH^p_s(A)\to CH^{g-p+s}_s(A)$ is an isomorphism.
\end{proposition}
\begin{proof}
\cite[Proposition 5.5]{beauvillearxiv}.
\end{proof}

\begin{proposition}\label{th1}
Cycles in $CH^p_{p-1}(A)$ are smash nilpotent for $p>1$.
\end{proposition}
\begin{proof}
Let $\alpha\in CH^p_{p-1}(A)$ be a numerically trivial cycle. Let $\sc{F}_A$ denote the Fourier transform. Let $B$ 
denote the abelian variety $A^{\times r}$. Observe that 
\[\sc{F}_{B}(\alpha^{\times r})=\sc{F}_A(\alpha)^{\times r}\]
Since $\sc{F}_A(\alpha)$ is a numerically trivial 1-cycle, it follows from proposition \ref{prop2} that it is 
smash nilpotent. In particular, using the above equality, for some $r>0$ we get that $\sc{F}_{B}(\alpha^{\times r})=0$ 
and since $\sc{F}$ is an isomorphism, $\alpha^{\times r}=0$ 
\end{proof}

Following \cite{fakhruddin} we denote by $\mbox{Griff}^{\,p}_s(A):=CH^p_s(A)\big/\sim_{alg}$.

\begin{theorem*}\cite[Theorem 4.4]{fakhruddin}
$\mbox{Griff}^{\,\,i}_2(A)$ ($i=3,4$) is infinite dimensional for the generic abelian variety $A$ of dimension 5.
\end{theorem*}
\begin{corollary*}\cite[Corollary 4.6]{fakhruddin}
$\mbox{Griff}^{\,\,g-1}_2(J)\neq0$ for the generic Jacobian $J$ of dimension $g\geq11$.
\end{corollary*}
\begin{corollary}
The above cycles are smash nilpotent. 
\end{corollary}
\begin{proof}
Follows from proposition \ref{th1} and theorem \ref{prop2}.
\end{proof}

\vskip 3mm
Let $C$ be a smooth projective curve. Let $A(J(C))$ denote the $\mbb{Q}$ algebra of cycles modulo 
algebraic equivalence on $J(C)$. The tautological subring of $A(J(C))$, see \cite{beauvilletaut}, is defined  
as the $\mbb{Q}$ subspace of $A(J(C))$ generated by the class of the curve (which is well 
defined since we are working modulo algebraic equivalence) under the operations : intersection product, Pontryagin 
product, pullback and pushforward under multiplication by integers. 

\begin{corollary}\label{cor2}
Numerical and smash equivalence coincide on the tautological subring of $A(J(C))$.
\end{corollary}
\begin{proof}
Let $g$ denote the genus of the curve. 
\noindent Let us denote the tautological subring by $R$. Since $R$ is stable under pullback by multiplication by integers, 
it has a decomposition as a sum of its Beauville components. The intersection product and Pontryagin 
product have the property 
\[A^p_s\cdot A^{p'}_{s'}\subset A^{p+p'}_{s+s'}\qquad\qquad A^p_s \ast A^{p'}_{s'}\subset A^{p+p'-g}_{s+s'}\]
We follow Beauville's notation and denote the class of the curve by $[C]=\sum_{i=0}^{g-1}C_{(i)}$ 
where $C_{(i)}\in A^{g-1}_i(J(C))$. Since $R$ is generated by the class of the curve, we get 
\[R=\bigoplus_p R^p=\bigoplus_p\bigoplus_{s\geq0}^pR^p_s=\Big(\bigoplus_pR^p_0\Big)\oplus\Big(\bigoplus_p\bigoplus_{s>0}^pR^p_s\Big)\]
The summand $\oplus_p\oplus_{s>0}^pR^p_s$ is an ideal in $R$, with respect to the intersection product and Pontryagin product, 
which we denote by $I$. 

Let $w^{g-d}:=(1/d!)C^{*d}\in A^{g-d}(J(C))$. The $N^i(w)$ are defined as the Newton polynomials in the classes $w^i$. 
For more details we refer the reader to \cite{beauvilletaut}. We will use the following results from \cite{beauvilletaut}, 
\begin{itemize}
	\item $R$ is generated as a $\mbb{Q}$-subalgebra, under the intersection product, by $N^i(w)$ (see the line 
		after \cite[Corollary 3.5]{beauvilletaut}). 
	\item In \cite[Corollary 3.4]{beauvilletaut} it is shown that $N^i(w)=-\sc{F}(C_{(i-1)})$ for $i\geq1$. In particular, 
		for $i\geq2$ one gets $N^i(w)\in I$.
\end{itemize}

Since $N^1(w)=\theta$, from the above it is clear that $\oplus_pR^p_0$ is the vector space spanned by $\langle\theta^p\rangle$. 
Thus, $0\neq\beta\in R$ is numerically trivial iff $\beta\in I$. Since $I$ is generated by $N^i(w)=-\sc{F}(C_{(i-1)})$ for $i\geq2$ 
we get that $\beta$ is smash nilpotent.
\end{proof}

\section{Solving the equations}\label{se}
Consider the following set of linear homogeneous polynomials in the $q$'s
\begin{enumerate}
	\item\label{s1} $l_1=\sum_{k=1}^{m}\left(\begin{array}{c}m\\k\end{array}\right)kq_{k}$
	\item\label{s2} $l_{2,2i}=\sum_{k=1}^{m}\left(\begin{array}{c}m\\k\end{array}\right)k^{2i}q_{k}$, for every $i$ in the set $\{1,\ldots,t\}$ (think of $t$ as $\lfloor\frac{g+1}{2}\rfloor$)
	\item\label{s3} $l_3=\sum_{k=n}^{m}\left(\begin{array}{c}m-n\\k-n\end{array}\right)q_{k}$
\end{enumerate}

It will suffice for us if we can find $m$'s and $q$'s such that the following systems have solutions
\begin{equation}\label{S1}l_1=0,l_{2,2i}=0,l_3\neq0\end{equation}
Consider the following collection of polynomials in one variable. 
\begin{enumerate}
	\item $r_i(x)=\sum_{k=1}^m\left(\begin{array}{c}m\\k\end{array}\right)k^{i}x^k$
	\item $x^{n}(1+x)^{m-n}=\sum_{k=n}^{m}\left(\begin{array}{c}m-n\\k-n\end{array}\right)x^k$
\end{enumerate}
If $T$ denotes the operator $\big(x\frac{d}{dx}\big)$ then 
\[r_i(x)= T^{i}(1+x)^m\]
Hence, it would suffice for us if we can show that the following polynomials are linearly independent for $2t<m$
\[\bigg{\{}x^n(1+x)^{m-n},\qquad r_1(x),\qquad r_{2i}(x)_{i=1,\ldots,t}\bigg{\}} \]
We claim this is true for $3\leq n\leq m$. 
If not, then $\exists$ a nontrivial linear dependence
\begin{equation}\label{e7}\beta_1T(1+x)^m+\sum_{i=1}^t\beta_{2i}T^{2i}(1+x)^m +ax^n(1+x)^{m-n}=0
\end{equation}

\noindent Define $m[j]:=m(m-1)\cdots (m-j)$. 

\begin{lemma}\label{lem1}
We have the following expression for $1\leq i\leq m$
\[T^i(1+x)^m=\sum_{j=1}^ic^i_jm[j-1]x^j(1+x)^{m-j}\]
where the coefficients $c^i_j$ satisfy the relation $c^{i+1}_j=jc^i_j+c^i_{j-1}$.
\end{lemma}
\begin{proof}
For $i=1$, $c^1_1=1$ and the statement is clear. We prove the statement by induction on $i$. In the following calculation, put $c^i_j=0$ for $i<j<1$. 
If $D$ denotes the operator $\frac{d}{dx}$, then
\begin{align*}
	DT^i(1+x)^m&=\sum_{j=1}^ic^i_jm[j-1]D\big(x^j(1+x)^{m-j}\big)\\
		&=\sum_{j=1}^ic^i_jm[j-1]\bigg(jx^{j-1}(1+x)^{m-j}+(m-j)x^j(1+x)^{m-j-1}\bigg)\\
	xDT^i(1+x)^m&=\sum_{j=1}^ic^i_jjm[j-1]x^j(1+x)^{m-j}+\sum_{j=1}^ic^i_jm[j]x^{j+1}(1+x)^{m-j-1}\\
	T^{i+1}(1+x)^m&=\sum_{j=1}^{i+1}\bigg(jc^i_j+c^i_{j-1}\bigg)m[j-1]x^j(1+x)^{m-j}
\end{align*}
\end{proof}
\begin{remark}
It is clear that the coefficients $c^i_j$ are independent of $m$.
\end{remark}

From the above lemma it follows that, for $i\leq m$, the highest power of $(1+x)$ which divides $T^i(1+x)^m$ is $m-i$. Looking at the least 
power of $(1+x)$ which divides all the terms in equation \eqref{e7}, we get that $\beta_i=0$ for $m-2i<m-n$, that is, 
for $i>\frac{n}{2}$. Thus, we may rewrite the relation as 
\begin{align}\label{e8}\beta_1T(1+x)^m&+\sum_{i=1}^{\lfloor\frac{n}{2}\rfloor}\beta_{2i}T^{2i}(1+x)^m =-ax^n(1+x)^{m-n}
\end{align}
This shows that \begin{equation}\label{e9}x^{n}\Bigg\vert\beta_1T(1+x)^m+\sum_{i=1}^{\lfloor\frac{n}{2}\rfloor}\beta_{2i}T^{2i}(1+x)^m\end{equation}

Next we will show that equation \eqref{e9} leads to a contradiction.

\begin{lemma}\label{soleq}
The following polynomials are linearly independent for $m>n\geq 3$, $2t<m$,
\[\bigg{\{}x^n(1+x)^{m-n},\qquad r_1(x),\qquad r_{2i}(x)_{i=1,\ldots,t}\bigg{\}} \]
\end{lemma}
\begin{proof}
Equating to 0 the coefficient of $x^k$, for $k\in\{1,2,\ldots,n-1\}$, in 
\[\beta_1T(1+x)^m+\sum_{i=1}^{\lfloor\frac{n}{2}\rfloor}\beta_{2i}T^{2i}(1+x)^m=\beta_1r_1(x)+\sum_{i=1}^{\lfloor\frac{n}{2}\rfloor}\beta_{2i}r_{2i}(x)\]
we get 
\begin{align}\label{e13}\beta_1k+\sum_{i=1}^{\lfloor\frac{n}{2}\rfloor}\beta_{2i}k^{2i}=0\qquad\qquad k\in\{1,2,\ldots,n-1\}
\end{align}
It is clear that all $\beta_i=0$ if $n-1\geq\lfloor\frac{n}{2}\rfloor+1$, that is, for $n\geq3$.
\end{proof}

\begin{lemma}\label{lem2}
For $1\leq i<n-3$, $\exists m\gg0$ such that the polynomial $(1+x)^{m-n}T^{i+3}(1+x)^n$ is not a rational linear combination 
of $r_{j}(x)$ for $j\in\{1,\ldots,n\}\setminus i$
\end{lemma}
\begin{proof}
Assume that 
\[(1+x)^{m-n}T^{i+3}(1+x)^n=\sum_{j\in\{1,\ldots,n\}\setminus i}\beta_jT^j(1+x)^m\]
Using lemma \ref{lem1} and comparing the highest power of $(1+x)$ dividing all terms in the above equation, we get $\beta_j=0$ for $j>i+3$. Thus,
\[(1+x)^{m-n}T^{i+3}(1+x)^n=\sum_{j\in\{1,\ldots,i+3\}\setminus i}\beta_jT^j(1+x)^m\]
Again using lemma \ref{lem1} and comparing the highest powers of $(1+x)$ on both sides of the above equality, we get, for $j\in\{i+3,i+2,i+1,i\}$, the following equations
\begin{align}
	c^{i+3}_{i+3}n[i+2]&=\beta_{i+3}c^{i+3}_{i+3}m[i+2]\nonumber\\
	c^{i+3}_{i+2}n[i+1]&=\beta_{i+2}c^{i+2}_{i+2}m[i+1]+\beta_{i+3}c^{i+3}_{i+2}m[i+1]\nonumber\\
	c^{i+3}_{i+1}n[i]&=\beta_{i+1}c^{i+1}_{i+1}m[i]+\beta_{i+2}c^{i+2}_{i+1}m[i]+\beta_{i+3}c^{i+3}_{i+1}m[i]\nonumber\\
	c^{i+3}_{i}n[i-1]&=\beta_{i+1}c^{i+1}_{i}m[i-1]+\beta_{i+2}c^{i+2}_{i}m[i-1]+\beta_{i+3}c^{i+3}_{i}m[i-1]\label{e21}
\end{align}
Solving for $\beta_{i+3}, \beta_{i+2}, \beta_{i+1}$ using the first three of the above equations we get 
\begin{align*}
	\beta_{i+3}&=\frac{n[i+2]}{m[i+2]}\qquad (\because c^i_i=1\qquad\forall i\geq 1)\\
	\beta_{i+2}&=c^{i+3}_{i+2}\Big(\frac{n[i+1]}{m[i+1]}-\frac{n[i+2]}{m[i+2]}\Big)\\
	\beta_{i+1}&=c^{i+3}_{i+1}\Big(\frac{n[i]}{m[i]}-\frac{n[i+2]}{m[i+2]}\Big)-c^{i+2}_{i+1}c^{i+3}_{i+2}\Big(\frac{n[i+1]}{m[i+1]}-\frac{n[i+2]}{m[i+2]}\Big)
\end{align*}
Substituting the above values into \eqref{e21}, we get 
\begin{align*}
c^{i+3}_i\frac{n[i-1]}{m[i-1]}=c^{i+1}_ic^{i+3}_{i+1}\Big(\frac{n[i]}{m[i]}-&\frac{n[i+2]}{m[i+2]}\Big)-c^{i+1}_ic^{i+2}_{i+1}c^{i+3}_{i+2}\Big(\frac{n[i+1]}{m[i+1]}-\frac{n[i+2]}{m[i+2]}\Big)\\
		&+c^{i+2}_ic^{i+3}_{i+2}\Big(\frac{n[i+1]}{m[i+1]}-\frac{n[i+2]}{m[i+2]}\Big)+c^{i+3}_i\frac{n[i+2]}{m[i+2]}
\end{align*}
For $r<s$ denote by $m[s,r]:=\frac{m[s]}{m[r]}$. With this notation the above equation is 
\begin{align*}
c^{i+3}_i=c^{i+1}_ic^{i+3}_{i+1}&\Big(\frac{n[i,i-1]}{m[i,i-1]}-\frac{n[i+2,i-1]}{m[i+2,i-1]}\Big)\\
	&-c^{i+1}_ic^{i+2}_{i+1}c^{i+3}_{i+2}\Big(\frac{n[i+1,i-1]}{m[i+1,i-1]}-\frac{n[i+2,i-1]}{m[i+2,i-1]}\Big)\\
		&+c^{i+2}_ic^{i+3}_{i+2}\Big(\frac{n[i+1,i-1]}{m[i+1,i-1]}-\frac{n[i+2,i-1]}{m[i+2,i-1]}\Big)+c^{i+3}_i\frac{n[i+2,i-1]}{m[i+2,i-1]}
\end{align*}
If we keep $i,n$ fixed and let $m\to \infty$, then the RHS is 0, while the LHS is a fixed positive integer, which is a contradiction.
\end{proof}


\begin{thebibliography}{Kun93}

\bibitem[Bea]{beauvillearxiv}
A.~Beauville, \emph{The action of $sl_2$ on abelian varieties arxiv:0805.1541}.

\bibitem[Bea86]{beauvillechow}
A.~Beauville, \emph{Sur l'anneau de chow d'une variété abélienne. (french) [the
  chow ring of an abelian variety]}, Math. Ann. \textbf{273} (1986), 647651.

\bibitem[Bea04]{beauvilletaut}
A.~Beauville, \emph{Algebraic cycles on jacobian varieties}, Compositio Math.
  \textbf{140} (2004), 683--688.

\bibitem[Fak96]{fakhruddin}
N.~Fakhruddin, \emph{Algebraic cycles on generic abelian varieties}, Compositio
  Math. \textbf{100} (1996), 101--119.

\bibitem[Ful97]{fulton}
W.~Fulton, \emph{Intersection theory}, Springer, 1997.

\bibitem[GS95]{schoen}
B.H. Gross and C.~Schoen, \emph{The modified diagonal cycle on the triple
  product of a pointed curve.}, Ann. Inst. Fourier (Grenoble) \textbf{45}
  (1995), 649--679.

\bibitem[Kim05]{kimura}
S.~I. Kimura, \emph{Chow groups are finite dimensional, in some sense}, Math.
  Ann. \textbf{331} (2005), 173--201.

\bibitem[Kle94]{kleiman}
S.~Kleiman, \emph{The standard conjectures}, Motives (Seattle, WA, 1991), Proc.
  Sympos. Pure Math., 55, Part 1, Amer. Math. Soc., Providence, RI (1994),
  3--20.

\bibitem[KS09]{ks}
B.~Kahn and R.~Sebastian, \emph{Smash-nilpotent cycles on abelian $3$-folds},
  Math. Res. Lett. \textbf{16} (2009), 1007--1010.

\bibitem[Kun93]{kunnemann}
K.~Kunnemann, \emph{A lefschetz decomposition for chow motives of abelian
  schemes}, Invent.\ Math. \textbf{113} (1993), 85--102.

\bibitem[She74]{shermenev}
A.~M. Shermenev, \emph{The motive of an abelian variety}, Funct. Anal. Appl.
  \textbf{8} (1974), 47--53.

\bibitem[Voe95]{voe}
V.~Voevodsky, \emph{A nilpotence theorem for cycles algebraically equivalent to
  zero.}, Internat. Math. Res. Notices \textbf{4} (1995), 187--198.

\bibitem[Voi94]{voisin}
C.~Voisin, \emph{Remarks on zero-cycles of self-products of varieties. moduli
  of vector bundles (sanda, 1994; kyoto, 1994)}, Lecture Notes in Pure and
  Appl. Math. \textbf{179} (1994), 265--285.

\end{thebibliography}
\end{document}